\documentclass{amsart}
\usepackage{amssymb, cite, colonequals, enumitem}
\usepackage[dvipsnames]{xcolor}
\usepackage[hidelinks]{hyperref}

\author{Masood Aryapoor}
\address[Masood Aryapoor]{Department of Business and Mathematics, M\"alar\-dalen  University, SE-631 05 Eskilstuna, Sweden}
\email{masood.aryapoor@mdu.se}

\author{Per B\"ack}
\address[Per B\"ack]{Department of Business and Mathematics, M\"alar\-dalen  University, SE-721  23  V\"aster\r{a}s, Sweden}
\email[corresponding author]{per.back@mdu.se}

\date\today

\newtheorem{corollary}{Corollary}

\newtheorem{proposition}{Proposition}
\newtheorem{theorem}{Theorem}

\theoremstyle{definition}
\newtheorem{definition}{Definition}
\newtheorem{example}{Example}

\theoremstyle{remark}
\newtheorem{remark}{Remark}

\DeclareMathOperator{\id}{id}
\DeclareMathOperator{\lc}{lc}
\newcommand{\fl}{\text{\upshape f\hspace{.56 pt}l}}
\newcommand{\degladd}{\text{\upshape deg\(_l\)-add}}
\newcommand{\degradd}{\text{\upshape deg\(_r\)-add}}

\begin{document}

\begin{abstract}
We introduce a broader class of nonassociative Ore extensions that
unifies and generalizes several earlier constructions, taking subadditivity of polynomial degree under multiplication as a guiding principle. We prove
generalizations of Hilbert's Basis Theorem for this class, showing
that left and right Noetherianity follow from corresponding
degree-reduction properties. Applications include flipped
nonassociative Ore extensions, mutations, cochain twists, and depolarizations of Poisson
Ore extensions.
\end{abstract}

\title[Hilbert's Basis Theorem for GNOEs]{Hilbert's Basis Theorem for generalized nonassociative Ore extensions}

\maketitle

\section{Introduction}

Hilbert's Basis Theorem, first established by Hilbert
\cite[Theorems~I and~II]{Hil90}, is a foundational result in commutative
algebra. It states that if a ring \(R\) is Noetherian, then the polynomial
ring \(R[x]\) is also Noetherian. The principal noncommutative
generalization of this result arises through Ore extensions, introduced by
Ore \cite{Ore33}, for which a first version of Hilbert's Basis Theorem was
established by Noether and Schmeidler \cite[Satz~III]{NS20}
(see e.g. \cite[Theorem~2.6]{GW04} for the general version).

In more recent work, several generalizations and analogs of Ore
extensions have been developed, extending these ideas beyond the
associative framework (see e.g.
\cite{AB25,BLOR26,BRS18,NOR18,NOR19,Oh06}). In many cases,
corresponding versions of Hilbert's Basis Theorem have also been proved
(see e.g. \cite{BLOR26,BLOR26b,BR23,BR24}).

In this article, we introduce a broader class of nonassociative Ore
extensions, called \emph{generalized nonassociative Ore extensions
(GNOEs)}, which unifies and generalizes several earlier constructions.
As in Ore's original work \cite{Ore33}, the guiding principle behind
our definition is to control polynomial degree under multiplication,
here through subadditivity. Our main result, \autoref{thm:hilbert},
generalizes Hilbert's Basis Theorem to GNOEs, showing that left and
right Noetherianity follow from corresponding degree-reduction
properties. This extends earlier results for nonassociative Ore
extensions and yields applications to flipped nonassociative Ore
extensions, mutations, cochain twists, and depolarizations of Poisson
Ore extensions.

The article is organized as follows:

In \autoref{sec:preliminaries}, we recall some background on
nonassociative rings and modules, as well as Ore extensions,
nonassociative Ore extensions, and flipped nonassociative Ore extensions.

In \autoref{sec:GNOE}, we introduce GNOEs and establish their basic
properties.

In \autoref{sec:hilbert}, we introduce the left and right
degree-reduction properties and prove \autoref{thm:hilbert} and its
two-sided consequence.

In \autoref{sec:applications}, we first give a leading-term criterion
that makes the degree-reduction properties easy to verify.  We then apply
the results to nonassociative and flipped nonassociative Ore extensions,
mutations, cochain twists, and Poisson constructions.

\section{Preliminaries}\label{sec:preliminaries}
\subsection{Nonassociative rings and modules}
We denote by \( \mathbb{N} \) the set of natural numbers, including zero. All rings and algebras other than Lie algebras are assumed to be
unital. Endomorphisms are assumed to preserve the identity element
when one exists. A \emph{nonassociative ring} is a ring that is not necessarily associative. Suppose that \( R \) is such a ring. The \emph{left}, \emph{middle}, and \emph{right nuclei} of \( R \) are defined as follows: \(
N_l(R) \colonequals \{ r \in R\colon (r,s,t) = 0 \text{ for all } s,t \in R \}\), \(N_m(R) \colonequals \{ s \in R\colon (r,s,t) = 0 \text{ for all } r,t \in R \}\), and \(N_r(R) \colonequals \{ t \in R\colon (r,s,t) = 0 \text{ for all } r,s \in R \}\), where \( (r,s,t) \colonequals (rs)t - r(st) \). Their intersection
\(N(R)\colonequals N_l(R)\cap N_m(R)\cap N_r(R)\)
is called the \emph{nucleus of \(R\)}. The \emph{commuter of \(R\)}
is \(C(R)\colonequals\{r\in R\colon rs=sr\text{ for all }s\in R\}\),
and the \emph{center of \(R\)} is
\(Z(R)\colonequals N(R)\cap C(R)\).

An element \(r\in R\) is called \emph{power-associative} if the subring of \(R\) generated by \(r\) is associative, or equivalently, if \(r^{n+1}\colonequals r^nr\) and \(r^mr^n=r^{m+n}\) for all \(m,n\in\mathbb{N}_{>0}\).

Let \(R\) be a nonassociative ring. A \emph{left \( R \)-module} is an abelian group \( M \) equipped with a biadditive map \( R \times M \to M \), \( (r,m) \mapsto rm \), called the \emph{module multiplication}. A subset \( B \) of \( M \) is called a \emph{basis} if every element \( m \in M \) can be uniquely expressed as a finite sum \( \sum_{b \in B} r_b b \), where \( r_b \in R \) and the sum has finite support. A left \( R \)-module with a basis is said to be \emph{free}. A \emph{submodule} of \( M \) is a subgroup of \( M \) that is closed under the module multiplication. Let \( N \) be a subset of \( M \). The intersection of all submodules of \( M \) that contain \( N \) is itself a submodule of \( M \), called the \emph{submodule of \( M \) generated by \( N \)}, denoted by \( {_R}\langle N \rangle \). In particular, \( {_R}\langle N \rangle \) is the smallest submodule of \( M \) that contains \( N \). If \( M = {_R}\langle N \rangle \) for some finite set \( N \), then \( M \) is said to be \emph{finitely generated by \( N \)}. The analogous definitions for right \( R \)-modules follow in a similar fashion.

A left or right \( R \)-module \( M \) is called \emph{Noetherian} if every submodule of \( M \) is finitely generated. Equivalently, this means that every ascending chain of submodules of \( M \) stabilizes, or that every nonempty family of submodules of \( M \) has a maximal element. \(R\) is called \emph{left Noetherian} if it is Noetherian as a left \(R\)-module, \emph{right Noetherian} if it is Noetherian as a right \(R\)-module, and \emph{Noetherian} if it is both left and right Noetherian. \(R\) is called \emph{left Artinian}
(respectively, \emph{right Artinian}) if every descending chain
of left ideals (respectively, right ideals) of \(R\) stabilizes.
It is called \emph{Artinian} if it is both left and right Artinian.

\subsection{Ore extensions}
We begin by recalling the definition of \emph{Ore extensions}.

\begin{definition}[Ore extension]\label{def:Ore} 
Let \(R\) be a ring. A pair \((S,x)\) is called a \emph{(left) Ore extension} of \(R\) if the following axioms hold:  
 \begin{enumerate}[label=(O\arabic*)]
 \item \(S\) is a ring extension of \(R\) and \(x\in S\);\label{it:ore1}
 \item \(S\) is associative;\label{it:ore2}
\item \(S\) is a free left \(R\)-module with basis \(\{1,x,x^2,\dots\}\);\label{it:ore3}
\item \(xR\subseteq R+Rx\).\label{it:ore4}
 \end{enumerate}
 \end{definition}

 Let \(R\) be a nonassociative ring with an endomorphism \(\sigma\). An additive map \(\delta\) on \(R\) is called a \emph{\(\sigma\)-derivation} if for all \(r,s\in R\), 
 \[\delta(rs)=\sigma(r)\delta(s)+\delta(r)s.\]

Given an associative ring \(R\) with an endomorphism \(\sigma\) and a \(\sigma\)-derivation \(\delta\), the \emph{ordinary generalized polynomial ring} \(R[X;\sigma,\delta]\) is the additive group \(R[X]\) of left polynomials in \(X\) with coefficients in \(R\), equipped with multiplication defined by the biadditive extension of the relations 
\begin{equation}
(rX^m)(sX^n)=\sum_{i\in\mathbb{N}}(r\pi_i^m(s))X^{i+n}.\label{eq:ore-mult}
\end{equation}
Here, the functions \(\pi_i^m\colon R\to R\) are defined as the sum of all \(\binom{m}{i}\) compositions of \(i\) copies of \(\sigma\) and \(m-i\) copies of \(\delta\). For example, \(\pi_2^3=\sigma\circ\sigma\circ\delta + \sigma\circ\delta\circ\sigma + \delta\circ\sigma\circ\sigma\). We define \(\pi_0^0\) as \(\id_R\) and set \(\pi_i^m=0_R\) whenever \(i>m\) or \(i<0\).

Any ordinary generalized polynomial ring \(R[X;\sigma,\delta]\) is an Ore extension \((S,X)\) of \(R\), and any Ore extension \((S,x)\) of \(R\) is isomorphic to an ordinary generalized polynomial ring \(R[X;\sigma,\delta]\) (for proofs, see e.g. \cite[Section 2.1]{Cohn95} or \cite[Propositions~3.2 and~3.3]{NOR18}). If \(\sigma=\id_R\) and \(\delta=0_R\), then \(R[X;\sigma,\delta]\) is the ordinary polynomial ring \(R[X]\).

\subsection{Nonassociative Ore extensions} In \cite{NOR18}, the authors observed that, for any nonassociative
ring \(R\) and additive maps \(\sigma,\delta\colon R\to R\)
satisfying \(\sigma(1)=1\) and \(\delta(1)=0\), the multiplication
in \eqref{eq:ore-mult} equips the additive group \(R[X]\) with a
nonassociative ring structure. To obtain an axiomatic description of these rings, they proposed
the following definition:

\begin{definition}[Nonassociative Ore extension]\label{def:n-Ore} 
Let \(R\) be a nonassociative ring. A pair \((S,x)\) is called a \emph{(left) nonassociative Ore extension} of \(R\) if the following axioms hold:  
 \begin{enumerate}[label=(NO\arabic*)]
 \item \(S\) is a ring extension of \(R\) and \(x\in S\);\label{it:n-ore1}
 \item \(x\in N_m(S)\cap N_r(S)\);\label{it:n-ore2}
\item \(S\) is a free left \(R\)-module with basis \(\{1,x,x^2,\dots\}\);\label{it:n-ore3}
\item \(xR\subseteq R+Rx\).\label{it:n-ore4}
 \end{enumerate}
 \end{definition}

 \begin{remark}Axiom \ref{it:n-ore2} ensures that the element \(x\) is power-associative, so that there is no ambiguity in writing \(x^n\) for the \(n\)-fold product of \(x\) with itself.
 \end{remark}

Given a nonassociative ring \(R\) with additive maps \(\sigma\) and  \(\delta\) that satisfy \(\sigma(1) = 1\) and \(\delta(1) = 0\), the \emph{generalized polynomial ring} \(R[X;\sigma,\delta]\) is the additive group \(R[X]\) of left polynomials in \(X\) with coefficients in \(R\), equipped with the multiplication defined by \eqref{eq:ore-mult}. By construction, any generalized polynomial ring \(R[X;\sigma,\delta]\) is a nonassociative Ore extension \((S,X)\) of \(R\) (see \cite[Proposition 3.2]{NOR18}), and any nonassociative Ore extension \((S,x)\) of \(R\) is isomorphic to a generalized polynomial ring \(R[X;\sigma,\delta]\) (see \cite[Proposition 3.3]{NOR18}). 

\subsection{Flipped nonassociative Ore extensions}\label{sec:flipped} Let \(A\) be a nonassociative algebra over an associative commutative ring \(K\). An \emph{involution of \(A\)} is a \(K\)-linear map  \(\alpha\colon A\to A\) that satisfies \(\alpha^2(a)=a\) and \(\alpha(ab)=\alpha(b)\alpha(a)\) for all \(a,b\in A\). A nonassociative algebra equipped with an involution \(*\) is called a \emph{nonassociative \(*\)-algebra}. 

In \cite[Theorem 1]{AB25}, the authors showed that any \emph{Cayley double} of a nonassociative \(*\)-algebra \(A\) can be described as a quotient of a nonassociative Ore extension of \(A\) equipped with a ``flipped'' multiplication \(\tau_n\colon A\times A\to A\) for \(n\in\mathbb{N}\), defined by 
\[\tau_n(a,b)\colonequals \begin{cases}ab&\text{if } n \text{ is even,}\\ ba&\text{if } n\text{ is odd.}\end{cases}\]
To formalize this idea, they introduced the notion of \emph{flipped nonassociative Ore extensions}:
\begin{definition}[Flipped nonassociative Ore extension]\label{def:flipped-n-Ore} 
Let \(R\) be a nonassociative ring. A pair \((S,x)\) is called a \emph{(left) flipped nonassociative Ore extension} of \(R\) if the following axioms hold:  
 \begin{enumerate}[label=(F\arabic*)]
 \item \(S\) is a ring extension of \(R\) and \(x\in S\);\label{it:flipped-n-ore1}
\item \(x\) is power-associative;\label{it:flipped-n-ore2}
\item \(S\) is a free left \(R\)-module with basis \(\{1,x,x^2,\dots\}\);\label{it:flipped-n-ore3}
 \item The following identities hold for all \(r,s\in R\) and  \(m,n\in\mathbb{N}\):
 
 \noindent\((rx^{m+1})(sx^n)=((rx^m)(\sigma(s)x^n))x+(rx^m)(\delta(s)x^n),\quad r(sx^n)=\tau_n(r,s)x^n\),
 where \(\sigma\) and \(\delta\) are additive maps such that \(xr=\sigma(r)x+\delta(r)\) for all \(r\in R\).\label{it:flipped-n-ore4}
 \end{enumerate}
 \end{definition}

\begin{remark}In \cite[Definition 5]{AB25}, Axiom \ref{it:flipped-n-ore2} was inadvertently omitted; it is, however, required for the validity of the results presented there.
\end{remark}
Given a nonassociative ring \(R\) with additive maps \(\sigma\) and  \(\delta\) that satisfy \(\sigma(1) = 1\) and \(\delta(1) = 0\), the \emph{flipped generalized polynomial ring} \(R[X;\sigma,\delta]^\fl\) is the additive group \(R[X]\) of left polynomials in \(X\) with coefficients in \(R\), equipped with multiplication defined by the biadditive extension of the relations 
\[(rX^m)(sX^n)=\sum_{i\in\mathbb{N}}\tau_n(r,\pi_i^m(s))X^{i+n}.\]
Any flipped generalized polynomial ring \(R[X;\sigma,\delta]^\fl\) is a flipped nonassociative Ore extension \((S,X)\) of \(R\) (see \cite[Proposition 1]{AB25}), and any flipped nonassociative Ore extension \((S,x)\) of \(R\) is isomorphic to a flipped generalized polynomial ring \(R[X;\sigma,\delta]^\fl\) (see \cite[Proposition 2]{AB25}).

\section{Generalized nonassociative Ore extensions}\label{sec:GNOE}
In this section, we introduce the notion of \emph{generalized nonassociative Ore extensions} and establish some of their basic properties.
\begin{definition}[Left generalized nonassociative Ore extension]\label{def:left-GNOE} 
Let \(R\) be a nonassociative ring. A pair \((S,x)\) is called a \emph{left generalized nonassociative Ore extension (left GNOE)} of \(R\) if the following axioms hold:  
 \begin{enumerate}[label=(LG\arabic*)]
 \item \(S\) is a ring extension of \(R\) and \(x\in S\); \label{it:lg_ore1}
 \item \(x\) is power-associative;\label{it:lg_ore2}
\item \(S\) is a free left \(R\)-module with basis \(\{1,x,x^2,\dots\}\);\label{it:lg_ore3}
\item \((Rx^m)(Rx^n)\subseteq R+Rx+\cdots+Rx^{m+n}\) for all \(m,n\in\mathbb{N}\).\label{it:lg_ore4}
 \end{enumerate}
 \end{definition}
 
Let \((S,x)\) be a left GNOE of a nonassociative ring \(R\). It follows from \ref{it:lg_ore3} that every nonzero element \(p\) in \(S\) can be uniquely written as 
\[
    p = r_0+r_1x+\cdots+r_nx^n,
\]
where \(r_0,\dots,r_n\in R\) with \(r_n\neq 0\). Using this unique representation of \(p\), we define the \emph{left degree} of \(p\) as \(\deg_l p \colonequals n\) and the \emph{left leading coefficient} of \(p\) as \(\lc_l(p) \colonequals r_n\). We use the conventions \(\deg_l 0 \colonequals -\infty\) and \(\lc_l(0) \colonequals 0\). 

The following proposition shows that \ref{it:lg_ore4} is equivalent
to subadditivity of polynomial degree under multiplication.

\begin{proposition}\label{prop:degree-condition}
Let $R$ be a nonassociative ring and \((S,x)\) be a pair that satisfies \ref{it:lg_ore1}, \ref{it:lg_ore2}, and \ref{it:lg_ore3}. Then \ref{it:lg_ore4}  holds if and only if \(\deg_l pq \leq \deg_l p +\deg_l q\) for all \(p,q\in S\). 
\end{proposition}
It follows from the definition that nonassociative Ore extensions form a subclass of the class of left GNOEs. Moreover, the following result holds: 
\begin{proposition}\label{prop:leftGNOE-NOE}
Let $(S,x)$ be a left GNOE of a nonassociative ring $R$. Then the following assertions hold:
\begin{enumerate}[label=\upshape(\roman*)]
\item \((S,x)\) is an Ore extension of \(R\) if and only if \(S\) is associative;
\item \((S,x)\) is a nonassociative Ore extension of \(R\) if and only if\\ \(x\in N_m(S)\cap N_r(S)\).
\end{enumerate}
\end{proposition}
As it turns out, flipped nonassociative Ore extensions also form a subclass of the class of left GNOEs:
\begin{proposition}
Any flipped nonassociative Ore extension of a nonassociative ring \(R\) is a left GNOE of \(R\).
\end{proposition}

\subsection{Right generalized nonassociative Ore extensions}

\autoref{def:left-GNOE} is formulated in terms of left \(R\)-modules. A corresponding notion can be defined using right \(R\)-modules in a parallel manner. 
\begin{definition}[Right generalized nonassociative Ore extension]\label{def:right-GNOE} 
Let \(R\) be a nonassociative ring. A pair \((S,x)\) is called a \emph{right generalized nonassociative Ore extension (right GNOE)} of \(R\) if the following axioms hold:  
 \begin{enumerate}[label=(RG\arabic*)]
 \item \(S\) is a ring extension of \(R\) and \(x\in S\); \label{it:rg_ore1}
 \item \(x\) is power-associative;\label{it:rg_ore2}
\item \(S\) is a free right \(R\)-module with basis \(\{1,x,x^2,\dots\}\);\label{it:rg_ore3}
\item \((x^mR)(x^nR)\subseteq R+xR+\cdots+x^{m+n}R\) for all \(m,n\in\mathbb{N}\).\label{it:rg_ore4}
 \end{enumerate}
 \end{definition}
 
Let \((S,x)\) be a right GNOE of a nonassociative ring \(R\). Every nonzero element \(p\in S\) can be uniquely expressed as
\[
    p = r_0 + x r_1 + \cdots + x^n r_n,
\]
where \(r_0,\dots,r_n\in R\) with \(r_n\neq 0\). We define the \emph{right degree} of \(p\) as \(\deg_r p \colonequals n\) and the \emph{right leading coefficient} of \(p\) as \(\lc_r(p) \colonequals r_n\). We adopt the conventions \(\deg_r 0 \colonequals -\infty\) and \(\lc_r(0) \colonequals 0\). We leave the corresponding statements and proofs for right GNOEs to the reader.

\begin{definition}[Generalized nonassociative Ore extension]\label{def:GNOE} 
A pair \((S,x)\) is called a \emph{generalized nonassociative Ore extension (GNOE)} of a nonassociative ring \(R\) if \((S,x)\) is both a left and right GNOE of \(R\).
\end{definition}

The following proposition states the precise condition under which the rings \(R[X;\sigma,\delta]\) and \(R[X;\sigma,\delta]^\fl\) are GNOEs of \(R\):

\begin{proposition}\label{prop:gnoe-noe}
    Let \(R\) be a nonassociative ring with additive maps \(\sigma\) and \(\delta\) that satisfy \(\sigma(1) = 1\) and \(\delta(1) = 0\). Then \(R[X;\sigma,\delta]\) (resp., \(R[X;\sigma,\delta]^\fl\)) is a GNOE of \(R\) if and only if \(\sigma\) is a bijection.
\end{proposition}

\begin{proof}
    Let \((S,x) \cong R[X;\sigma,\delta]\) be a GNOE. For any \(r\in R\), we have \(\sigma(r)=0\) if and only if \(-\delta(r)+xr=0\) if and only if \(r=0\) because \(1,x\) are right linearly independent over \(R\). This shows that \(\sigma \) is injective. Since \(\{1,x,x^2,\dots\}\) is a basis for \(S\) as a right \(R\)-module, we have \(rx\in \sum_n x^n R\) for every \(r\in R\).  Since \(\sigma\) is injective, \(rx = r_0+xr_1\) for some \(r_0,r_1\in R\). It follows that \(r=\sigma(r_1)\), that is, \(\sigma\) is onto. This completes the proof of the forward direction. 

    Conversely, suppose that \(\sigma \) is a bijection. Since \((S,x)\) is a nonassociative Ore extension of \(R\), it is a left GNOE of \(R\). Thus it remains to show that
\((S,x)\) is a right GNOE of \(R\). Since \(\sigma\) is onto, we have \(S=\sum_n x^nR\). To prove \ref{it:rg_ore3}, it remains to show that \(S\) is a free right \(R\)-module with basis \(\{1,x,x^2,\dots\}\). For any \[p = r_0+xr_1+\cdots+x^nr_n\in S,\] where \(r_n\neq 0\), the left leading coefficient of \(p\) is \(\sigma^n(r_n)\neq 0\). This shows that \(\{1,x,x^2,\dots\}\) is a basis for \(S\) as a right \(R\)-module. It is easy to verify that \(\deg_l p =\deg_r p\) for any \(p\in S\), from which \ref{it:rg_ore4} follows. This completes the proof of the reverse direction. 

    The proof for flipped nonassociative Ore extensions is similar and is therefore left to the reader. 
\end{proof}

\section{A generalization of Hilbert's Basis Theorem}\label{sec:hilbert}
Let \((S,x)\) be a left GNOE of a nonassociative ring \(R\), and let \(I=\langle p_1,\ldots,p_n \rangle_S\) for some \(p_1,\ldots,p_n\in S\). Every element of \(I\) is a finite sum of terms obtained by
successive right multiplications of a generator by elements of \(S\). We denote by \(I^\degladd(p_1,\dots,p_n)\) (or simply
\(I^\degladd\) when no confusion can arise) the set of all
\emph{left degree-additive elements of \(I\) relative to
\(p_1,\dots,p_n\)}, that is, those \(u\in I\) that have a
representation
\[
u=\sum_i\sum_j
(\cdots((p_iu_{ij1})u_{ij2})\cdots)u_{ijm},
\]
for some \(m\geq1\) and \(u_{ijk}\in S\)
(\(1\leq i\leq n\), \(1\leq j,k\leq m\)), such that
\(
\deg_l u
=\max_i\left(\deg_l p_i+\max_j\sum_k\deg_l u_{ijk}\right).
\)

\begin{definition}[Left degree-reduction property]\label{def:left_division} 
A left GNOE \((S,x)\) of a nonassociative ring \(R\) is said to
satisfy the \emph{left degree-reduction property} if the following
condition holds:

 \begin{enumerate}[label=(LDR)]
 \item For all \(p_1,\dots,p_n\in S\) and nonzero \(q\in S\), if \(\deg_l q \geq \max_i\deg_l p_i\) and \(\lc_l(q)\in \langle \lc_l(p_1),\ldots,\lc_l(p_n)\rangle_R\), then there exists \(u\in \langle p_1,\ldots,p_n\rangle_S^\degladd\) such that \(\deg_l\left(q-u\right) <\deg_l q\).
 \label{it:ldivision}
 \end{enumerate}
 \end{definition}
 
\begin{remark}\label{rem:LEDA-monomials}
To show that a left GNOE satisfies the left degree-reduction property, it is enough to verify Condition \ref{it:ldivision} for the case where \(p_1,\dots,p_n,q\) are monomials. Indeed, replacing the leading monomials by the original
polynomials changes the resulting degree-additive expression
only by terms of lower degree.
\end{remark}

\subsection{Hilbert's Basis Theorem for GNOEs}
It is straightforward to formulate the degree-reduction property and Hilbert's Basis Theorem for right GNOEs. For the reader's convenience, we provide a brief description. Let \((S,x)\) be a right GNOE of a nonassociative ring \(R\), and suppose that \(I={}_S\langle p_1,\ldots,p_n \rangle\) for some \(p_1,\ldots,p_n\in S\). Every element of \(I\) is a finite sum of terms obtained by
successive left multiplications of a generator by elements of \(S\). We denote by \(I^\degradd(p_1,\dots,p_n)\) (or simply
\(I^\degradd\) when no confusion can arise) the set of all
\emph{right degree-additive elements of \(I\) relative to
\(p_1,\dots,p_n\)}, that is, those \(u\in I\) that have a
representation
\[
u=\sum_i\sum_j
u_{ijm}(\cdots(u_{ij2}(u_{ij1}p_i))\cdots),
\]
for some \(m\geq1\) and \(u_{ijk}\in S\)
(\(1\leq i\leq n\), \(1\leq j,k\leq m\)), such that
\(
\deg_r u
=\max_i\left(\deg_r p_i+\max_j\sum_k\deg_r u_{ijk}\right).
\)

\begin{definition}[Right degree-reduction property]\label{def:right_division} 
A right GNOE \((S,x)\) of a nonassociative ring \(R\) is said to satisfy the \emph{right degree-reduction property} if the following condition holds: 
 \begin{enumerate}[label=(RDR)]
 \item For all \(p_1,\dots,p_n\in S\) and nonzero \(q\in S\), if \(\deg_r q \geq \max_i\deg_r p_i\) and \(\lc_r(q)\in {}_R\langle \lc_r(p_1),\ldots,\lc_r(p_n)\rangle\), then there exists \(u\in {}_S\langle p_1,\ldots,p_n\rangle^\degradd\) such that \(\deg_r\left(q-u\right) <\deg_r q\).
 \label{it:rdivision}
 \end{enumerate}
 \end{definition}

We now turn to the main result of this article. 

\begin{theorem}[Hilbert's Basis Theorem for GNOEs]
\label{thm:hilbert}
Let \(R\) be a nonassociative ring, \(S\) a ring extension of \(R\),
and \(x\in S\). Then the following assertions hold:
\begin{enumerate}[label=\upshape(\roman*)]
\item\label{it:hilbert-left-GNOE}
If \((S,x)\) is a left GNOE of \(R\) satisfying \ref{it:ldivision} and
\(R\) is right Noetherian, then \(S\) is right Noetherian.
\item\label{it:hilbert-right-GNOE}

If \((S,x)\) is a right GNOE of \(R\) satisfying \ref{it:rdivision} and
\(R\) is left Noetherian, then \(S\) is left Noetherian.
\end{enumerate}
\end{theorem}

\begin{proof}
We first prove \ref{it:hilbert-left-GNOE}. Assume that \(R\) is
right Noetherian. We wish to show that any nonzero right ideal
\(I\) of \(S\) is finitely generated. For each \(m\in\mathbb N\),
set \(J_m\colonequals\{\lc_l(p)\colon p\in I,\ \deg_l p\leq m\}\).
We first show that \(J_m\) is a right ideal of \(R\).

Let \(p_1,\ldots,p_n\in I\) be nonzero polynomials of degree at most
\(m\), and let \[0\neq r\in
\langle\lc_l(p_1),\ldots,\lc_l(p_n)\rangle_R.\]
Set \(d\colonequals\max_i\deg_l p_i\). Applying
\ref{it:ldivision} to \(p_1,\ldots,p_n\) and
\(q\colonequals rx^d\) gives
\[u\in\langle p_1,\ldots,p_n\rangle_S^\degladd\]
such that \(\deg_l(q-u)<d\). Hence \(u\in I\),
\(\deg_l u=d\leq m\), and \(\lc_l(u)=r\).
Since \(0\in J_m\), this proves that
\[
\langle\lc_l(p_1),\ldots,\lc_l(p_n)\rangle_R
\subseteq J_m.
\]
Since \(J_m\) contains the right ideal generated by every finite
subset of \(J_m\), it is itself a right ideal of \(R\).

Clearly, \(J_0\subseteq J_1\subseteq J_2\subseteq\cdots\).
Since \(R\) is right Noetherian, this chain stabilizes.
Choose \(d\in\mathbb N\) such that \(J_m=J_d\) for every
\(m\geq d\). For each \(0\leq m\leq d\), choose nonzero
generators \(r_{m,1},\ldots,r_{m,n_m}\) of \(J_m\), with an
empty family if \(J_m=0\). Choose \(p_{m,j}\in I\) such that
\(\lc_l(p_{m,j})=r_{m,j}\) and \(\deg_l p_{m,j}\leq m\).
Set
\[
I'\colonequals
\langle p_{m,j}\colon
0\leq m\leq d,\ 1\leq j\leq n_m\rangle_S.
\]
Then \(I'\subseteq I\), and \(I'\) is finitely generated.
We prove that \(I\subseteq I'\) by induction on degree.

Base case (\(m=0\)): Let \(q\in I\) with \(\deg_l q=0\).
Since \(q=\lc_l(q)\) and \(p_{0,i}=r_{0,i}\), we have
\(
q\in J_0
=\langle r_{0,1},\ldots,r_{0,n_0}\rangle_R
\subseteq I'.
\)

Induction step (\(m+1\)): Assume that every element of \(I\)
of degree at most \(m\) belongs to \(I'\), and let \(q\in I\)
with \(\deg_l q=m+1\). Set \(e\colonequals\min(m+1,d)\).
Since \(J_{m+1}=J_e\), we have
\(\lc_l(q)\in\langle r_{e,1},\ldots,r_{e,n_e}\rangle_R\).
Moreover, \(\deg_l p_{e,j}\leq e\leq m+1\).
Applying \ref{it:ldivision} to
\(p_{e,1},\ldots,p_{e,n_e}\) and \(q\), we obtain
\(u\in I'\) such that \(\deg_l(q-u)<m+1\).
Since \(q,u\in I\), we have \(q-u\in I\).
By the induction hypothesis, \(q-u\in I'\).
Hence \(q=(q-u)+u\in I'\), proving \(I=I'\).
Since \(I'\) is finitely generated, we conclude that
\(I\) is finitely generated.

For \ref{it:hilbert-right-GNOE}, the pair \((S^{\mathrm{op}},x)\)
is a left GNOE of \(R^{\mathrm{op}}\). Reversing multiplication
identifies its left degrees, leading coefficients, and generated
right submodules with the corresponding right degrees, leading
coefficients, and generated left submodules for \((S,x)\).
Thus \ref{it:rdivision} for \((S,x)\) gives
\ref{it:ldivision} for \((S^{\mathrm{op}},x)\).
Since \(R^{\mathrm{op}}\) is right Noetherian,
\ref{it:hilbert-left-GNOE} shows that \(S^{\mathrm{op}}\)
is right Noetherian, equivalently that \(S\) is left Noetherian.
\end{proof}

The following corollary follows immediately from \autoref{thm:hilbert}:

\begin{corollary}
\label{cor:hilbert-for-GNOE}
Let \((S,x)\) be a GNOE of a nonassociative ring \(R\) satisfying
\ref{it:ldivision} and \ref{it:rdivision}.
If \(R\) is Noetherian, then so is \(S\).
\end{corollary}
 
\section{Applications}
\label{sec:applications}
In this section, we apply \autoref{thm:hilbert} to several classes
of nonassociative, noncommutative polynomial rings. We begin with
\autoref{prop:leading-term}, which gives a sufficient condition
for \ref{it:ldivision} and \ref{it:rdivision}, making Noetherianity
easy to verify.

For a nonassociative ring \(R\), let \(R[x]\) denote the additive
group of left polynomials in \(x\) with coefficients in \(R\). For a nonassociative ring \(S\) with additive group \(R[x]\), we use \(\cdot\)
to denote multiplication in \(S\) when needed for clarity.

\begin{proposition}
\label{prop:leading-term}
Let \(R\) be a nonassociative ring with an automorphism \(\sigma\),
and let \(S\) be a ring extension of \(R\) with additive group \(R[x]\).
Set \(F_d\colonequals\bigoplus_{i=0}^{d}Rx^i\) for \(d\in\mathbb N\)
and \(F_{-1}\colonequals0\). Suppose that, for all \(r,s\in R\)
and \(m,n\in\mathbb N\),
\[
x^m x^n=x^{m+n},\qquad
(rx^m)(sx^n)-\bigl(r\sigma^m(s)\bigr)x^{m+n}\in F_{m+n-1}.
\]
Then \((S,x)\) is a GNOE of \(R\) satisfying
\ref{it:ldivision} and \ref{it:rdivision}.
If \(R\) is right, respectively left, Noetherian, then so is \(S\).
\end{proposition}

\begin{proof}
Axioms \ref{it:lg_ore1} and \ref{it:rg_ore1} hold by assumption,
while \(x^m x^n=x^{m+n}\) proves \ref{it:lg_ore2} and
\ref{it:rg_ore2}. The multiplication hypothesis gives
\[
r\cdot x^n\equiv rx^n\pmod{F_{n-1}},\qquad
x^n\cdot r\equiv\sigma^n(r)x^n\pmod{F_{n-1}}.
\]
Hence, for \(p=\sum_{i=0}^{d}r_ix^i\), both
\(p-r_d\cdot x^d\) and \(p-x^d\cdot\sigma^{-d}(r_d)\)
belong to \(F_{d-1}\). Induction gives
\(F_d=\sum_{i=0}^{d}R\cdot x^i=\sum_{i=0}^{d}x^i\cdot R\).

If \(\sum_{i=0}^{d}r_i\cdot x^i=0\), comparison of the coefficient
of \(x^d\) gives \(r_d=0\), and induction gives \(r_i=0\) for all \(i\).
If \(\sum_{i=0}^{d}x^i\cdot r_i=0\), the same comparison gives
\(\sigma^d(r_d)=0\), so injectivity of \(\sigma\) gives the same
conclusion. Thus \ref{it:lg_ore3} and \ref{it:rg_ore3} hold.
The multiplication hypothesis and biadditivity give
\(F_mF_n\subseteq F_{m+n}\), proving \ref{it:lg_ore4} and
\ref{it:rg_ore4}. Both degrees equal the degree of the underlying
left polynomial. For \(p\neq0\) of degree \(d\), we also have
\begin{equation}
\label{eq:twisted-leading-coefficients}
p\equiv\lc_l(p)x^d=\sigma^d\bigl(\lc_r(p)\bigr)x^d
\pmod{F_{d-1}}.
\end{equation}

For \ref{it:ldivision}, let \(p_1,\ldots,p_n,q\) satisfy its hypotheses.
Omit zero generators and set
\(I\colonequals\langle p_1,\ldots,p_n\rangle_S\),
\(d_i\colonequals\deg_l p_i\), \(d\colonequals\deg_l q\),
\(r_i\colonequals\lc_l(p_i)\), and \(r\colonequals\lc_l(q)\).
Using the earlier description of generated right submodules, write
\[
r=\sum_i\sum_j
(\cdots((r_it_{ij1})t_{ij2})\cdots)t_{ij\ell},
\]
with finite sums and a common number \(\ell\) of multipliers,
inserting factors \(1\) when necessary.
Omit zero summands and set
\(u_{ijk}\colonequals\sigma^{-d_i}(t_{ijk})\). Define
\[
u\colonequals\sum_i\sum_j
\bigl((\cdots((p_i\cdot u_{ij1})\cdot u_{ij2})\cdots)
\cdot u_{ij\ell}\bigr)\cdot x^{d-d_i}.
\]
Successive applications of the multiplication hypothesis give
\[
u\equiv\sum_i\sum_j
\bigl((\cdots((r_it_{ij1})t_{ij2})\cdots)t_{ij\ell}\bigr)x^d
=rx^d\equiv q\pmod{F_{d-1}}.
\]
Thus \(\deg_l u=d\). Each remaining \(u_{ijk}\) is a nonzero
constant, so every summand in this representation satisfies
\[
\deg_l p_i+\sum_{k=1}^{\ell}\deg_l u_{ijk}
+\deg_l x^{d-d_i}
=d_i+0+(d-d_i)=d.
\]
Hence \(u\in I^\degladd(p_1,\ldots,p_n)\) and
\(\deg_l(q-u)<d\), proving \ref{it:ldivision}.

For \ref{it:rdivision}, the multiplication hypothesis gives, for
\(r,s\in R\) and \(m,n\in\mathbb N\),
\[
\begin{aligned}
(x^m\cdot r)\cdot(x^n\cdot s)
&\equiv
\bigl(\sigma^m(r)\sigma^{m+n}(s)\bigr)x^{m+n}\pmod{F_{m+n-1}}\\
&\equiv
x^{m+n}\cdot\bigl(\sigma^{-n}(r)s\bigr)
\pmod{F_{m+n-1}},
\end{aligned}
\]
where the second congruence follows from multiplicativity of
\(\sigma\) and \eqref{eq:twisted-leading-coefficients}.
Using the right \(R\)-basis as a left \(R^{\mathrm{op}}\)-basis,
\(S^{\mathrm{op}}\) therefore satisfies the same hypotheses over
\(R^{\mathrm{op}}\), with \(\sigma\) replaced by \(\sigma^{-1}\).
Applying the preceding proof of \ref{it:ldivision} to
\(S^{\mathrm{op}}\) proves \ref{it:rdivision} for \(S\).

The Noetherianity assertions follow from \autoref{thm:hilbert}.
\end{proof}

\subsection{Ore extensions}
An application of \autoref{thm:hilbert} gives the following result.
For \(R[x;\sigma,\delta]\), see also B\"ack and Richter
\cite[Theorem~3.7]{BR24}.

\begin{proposition}
\label{prop:nonORe-left-division}
Let \(R\) be a nonassociative ring with an additive surjection
\(\sigma\) and an additive map \(\delta\) satisfying
\(\sigma(1)=1\) and \(\delta(1)=0\).
Then \(R[x;\sigma,\delta]\) and \(R[x;\sigma,\delta]^\fl\)
satisfy \ref{it:ldivision}. If \(R\) is right Noetherian, then so are
\(R[x;\sigma,\delta]\) and \(R[x;\sigma,\delta]^\fl\).
\end{proposition}

\begin{proof}
Let \(S\) denote either nonassociative ring, and let
\(p_1,\ldots,p_n,q\) satisfy the hypotheses of \ref{it:ldivision}.
Omit zero generators and set
\(I\colonequals\langle p_1,\ldots,p_n\rangle_S\),
\(d_i\colonequals\deg_l p_i\), \(d\colonequals\deg_l q\),
\(r_i\colonequals\lc_l(p_i)\), and \(r\colonequals\lc_l(q)\).
Write
\[
r=\sum_i\sum_j
(\cdots((r_it_{ij1})t_{ij2})\cdots)t_{ij\ell},
\]
omitting zero summands and inserting factors \(1\) as needed.
Surjectivity gives \(u_{ijk}\in R\) with
\(\sigma^{d_i}(u_{ijk})=t_{ijk}\).

In both constructions, the coefficient of \(x^{d_i}\) in
\(p_i\cdot t\), for \(t\in R\), is \(r_i\sigma^{d_i}(t)\),
and right multiplication
by a power of \(x\) shifts the degrees without changing coefficients.
Consequently,
\[
u\colonequals\sum_i\sum_j
\bigl((\cdots((p_i\cdot u_{ij1})\cdot u_{ij2})\cdots)
\cdot u_{ij\ell}\bigr)\cdot x^{d-d_i}
\]
has degree \(d\) and left leading coefficient \(r\).
Each summand has generator degree plus multiplier degrees equal to
\(d_i+0+(d-d_i)=d\). Hence
\(u\in I^\degladd(p_1,\ldots,p_n)\) and \(\deg_l(q-u)<d\).

The Noetherianity assertion follows from
\ref{it:hilbert-left-GNOE} in \autoref{thm:hilbert}.
\end{proof}

Surjectivity cannot simply be omitted. For a field \(K\), the
endomorphism of \(R\colonequals K[y]\) fixing \(K\) pointwise
and satisfying \(\sigma(y)=y^2\) gives a ring
\(R[x;\sigma,0_R]=R[x;\sigma,0_R]^\fl\) which is neither left
nor right Noetherian; see Goodearl and Warfield
\cite[Exercise~2P(a)]{GW04}.

If \(\sigma\) is an automorphism, the multiplication
\eqref{eq:ore-mult} satisfies the hypotheses of
\autoref{prop:leading-term}. Hence \(R[x;\sigma,\delta]\) satisfies
both \ref{it:ldivision} and \ref{it:rdivision}. In particular, if
\(R\) is left Noetherian, then \(R[x;\sigma,\delta]\) is left
Noetherian by \autoref{thm:hilbert}. This Noetherianity conclusion is
also a special case of \cite[Theorem~45]{BLOR26}, obtained by taking
\(M=R\), \(\sigma_M=\sigma_R=\sigma\), and
\(\delta_M=\delta_R=\delta\).

The automorphism assumption cannot in general be replaced by the
assumption that \(\sigma\) is merely an additive bijection satisfying
\(\sigma(1)=1\); see \cite[Example~3.11]{BR24}. The corresponding left Noetherianity statement does not hold for
flipped nonassociative Ore extensions:

\begin{example}[Failure of left Noetherianity]
\label{ex:flipped-sidedness}
Let \(R\) be a nonassociative ring that is left Noetherian but
not right Noetherian. Then there exists an ascending chain of
right ideals \(I_1\subsetneq I_2\subsetneq I_3\subsetneq\cdots\) in \(R\).
Set \(R[x]^\fl\colonequals R[x;\id_R,0_R]^\fl\).
For \(i\geq1\), set
\(J_i\colonequals I_ix+R[x]^\fl x^2\).
Each \(J_i\) is a left ideal of \(R[x]^\fl\). Then \(J_1\subsetneq J_2\subsetneq J_3\subsetneq \cdots\)
is an ascending chain of left ideals of \(R[x]^\fl\).
Hence \(R[x]^\fl\) is not left Noetherian.
\end{example}

\subsection{Mutations}

Scalar mutations were introduced by Albert
\cite[Chapter~V]{Alb48} in his study of quasi-associative algebras;
see also Price \cite{Pri51} and McCrimmon \cite{McC66}.
Quasi-associative algebras play an important role in the classification
of simple noncommutative Jordan algebras; see Schafer
\cite[Corollary~3]{Sch55} and Oehmke \cite[Theorem~4.2]{Oeh58}.

For a nonassociative ring \(R\) and \(\lambda\in Z(R)\), the
\emph{\(\lambda\)-mutation} of \(R\), denoted by \(R^{(\lambda)}\),
is the nonassociative ring with additive group \(R\) and multiplication
\[
r\star_\lambda s\colonequals\lambda(rs)+(1-\lambda)(sr).
\]
In particular, \(R^{(0)}=R^{\mathrm{op}}\) and \(R^{(1)}=R\).
If \(R\) is associative and \(2\) is invertible, then
\(R^{(1/2)}\) is the associated Jordan ring.

The following example shows that mutations need not preserve
Noetherianity.

\begin{example}[A non-Noetherian mutation]
\label{ex:mutation-noetherianity}
Let \(L\) be an infinite-dimensional simple Lie algebra over a field
\(K\) of characteristic zero. The classical unitalization of \(L\)
(see e.g. Schafer \cite[p.~11]{Sch66}) is obtained by equipping
\(A\colonequals K1\oplus L\) with multiplication
\[
(\alpha1+\ell_1)(\beta1+\ell_2)
\colonequals\alpha\beta1+\alpha\ell_2+\beta\ell_1+[\ell_1,\ell_2]
\]
for \(\alpha,\beta\in K\) and \(\ell_1,\ell_2\in L\).
For \(\alpha\neq0\) and \(\ell\in L\), the identity
\((\alpha1+\ell)(\alpha1-\ell)
=(\alpha1-\ell)(\alpha1+\ell)=\alpha^2 1\)
shows that every proper one-sided ideal of \(A\) is contained
in \(L\), where it is a Lie ideal.
Thus the only left or right ideals of \(A\) are \(0,L,A\),
so \(A\) is Noetherian.

However, \(\ell_1\star_{1/2}\ell_2=0\) for all
\(\ell_1,\ell_2\in L\), so every \(K\)-subspace of \(L\)
is a two-sided ideal of \(A^{(1/2)}\).
Since \(L\) is infinite-dimensional, \(A^{(1/2)}\) is not Noetherian.
\end{example}

\begin{proposition}
\label{prop:mutations}
Let \(R\) be a nonassociative ring, \(\delta\colon R\to R\) an additive
map with \(\delta(1)=0\), and \(\lambda\in Z(R)\) such that
\(\delta(\lambda r)=\lambda\delta(r)\) for every \(r\in R\).
Set \(S\colonequals R[x;\id_R,\delta]\).
Then \((S^{(\lambda)},x)\) is a GNOE of \(R^{(\lambda)}\)
satisfying \ref{it:ldivision} and \ref{it:rdivision}.
If \(R^{(\lambda)}\) is right, respectively left, Noetherian,
then so is \(S^{(\lambda)}\).
\end{proposition}

\begin{proof}
The compatibility condition gives \(\delta(\lambda)=0\) and
\(\delta^j(\lambda r)=\lambda\delta^j(r)\) for every \(j\in\mathbb N\).
The multiplication in \(S\) satisfies
\[
(rx^m)(sx^n)=
\sum_{j=0}^{m}\binom{m}{j}
\bigl(r\delta^j(s)\bigr)x^{m+n-j}.
\]
Together with \(\lambda\in Z(R)\), this gives
\(\lambda p=p\lambda\) and
\((\lambda p)q=\lambda(pq)=p(\lambda q)\) for all \(p,q\in S\).
Thus \(\lambda\) belongs to all three nuclei and commutes with
every element of \(S\), so \(\lambda\in Z(S)\).

With \(F_d\) as in \autoref{prop:leading-term},
\[
(rx^m)\star_\lambda(sx^n)
\equiv
\bigl(\lambda(rs)+(1-\lambda)(sr)\bigr)x^{m+n}
\pmod{F_{m+n-1}},
\]
and \(x^m\star_\lambda x^n=x^{m+n}\).
Applying \autoref{prop:leading-term} with coefficient ring
\(R^{(\lambda)}\) and \(\sigma\colonequals\id_{R^{(\lambda)}}\) completes the proof. 
\end{proof}

If \(\delta\) is a derivation, the compatibility condition in
\autoref{prop:mutations} is equivalent to \(\delta(\lambda)=0\).
If \(R\) is commutative, then \(R^{(\lambda)}=R\); if \(R\) is
also Noetherian, then so is \(S^{(\lambda)}\).
The GNOE assertion need not extend to mutations of
\(R[x;\sigma,\delta]\) with an automorphism \(\sigma\), since
\ref{it:lg_ore3} and \ref{it:rg_ore3} may fail.

\begin{example}[Weyl and Heisenberg mutations]
\label{ex:weyl-heisenberg-mutations}
Let \(K\) be a field and \(\lambda\in K\).
By \autoref{prop:mutations}, the \(\lambda\)-mutations of the
Weyl algebra \(K[y][x;\id_{K[y]},\partial_y]\) and the Heisenberg
enveloping algebra \(K[y,z][x;\id_{K[y,z]},z\partial_y]\) are
GNOEs over \(K[y]\) and \(K[y,z]\), respectively, satisfying
\ref{it:ldivision} and \ref{it:rdivision}; hence they are Noetherian.
\end{example}

Let \(S\) denote either the Weyl algebra or the Heisenberg enveloping algebra in \autoref{ex:weyl-heisenberg-mutations}, with its associative multiplication. Set \(c\colonequals xy-yx\), so \(c=1\) or \(c=z\), respectively. If \(\lambda\notin\{0,1\}\), then
\(
((yx)\star_\lambda x)\star_\lambda y
-(yx)\star_\lambda(x\star_\lambda y)
=\lambda(1-\lambda)c^2\neq0.
\)
Thus \(x\notin N_m(S^{(\lambda)})\).
By \autoref{prop:leftGNOE-NOE}, these presentations are not
nonassociative Ore extensions, nor are they flipped ones,
since their coefficient rings are commutative.

\subsection{Cochain twists}

Albuquerque and Majid \cite[Corollary~2.4]{AM99} construct
nonassociative cochain twists of group algebras.
For an associative commutative ring \(K\), let \(K^\times\)
denote its group of units.
Given a group \(G\) with identity \(e\), a normalized
\(2\)-cochain is a map \(F\colon G\times G\to K^\times\)
satisfying \(F(e,g)=F(g,e)=1\) for every \(g\in G\).
The additive group \(\bigoplus_{g\in G}Kg\) is equipped with
multiplication defined by the biadditive extension of
\((\alpha g)\cdot(\beta h)\colonequals\bigl(\alpha \beta F(g,h)\bigr)(gh)\)
for \(\alpha,\beta\in K\) and \(g,h\in G\).
Their construction is over a field; the same formula defines
a nonassociative \(K\)-algebra.

We take \(G\colonequals\mathbb Z^2\) and
\(F((i,m),(j,n))\colonequals\gamma_m^{\,j}\), where
\(\gamma_m\in K^\times\) for \(m\in\mathbb Z\) and \(\gamma_0=1\).
For \(\gamma_m=\gamma^m\), \(\gamma\in K^\times\), restriction
to \(\mathbb N^2\) gives the \emph{quantum plane}, the associative
\(K\)-algebra generated by \(x\) and \(y\) with defining relation \(xy=\gamma yx\).
Allowing arbitrary units \(\gamma_m\) gives natural nonassociative
generalizations. The following proposition shows that these
are Noetherian whenever \(K\) is Noetherian.

\begin{proposition}
\label{prop:cochain}
Let \(K\) be an associative commutative Noetherian ring, and let
\(\gamma_m\in K^\times\) for every \(m\in\mathbb N\), with \(\gamma_0=1\).
Let \(S\) have additive group \(K[y][x]\) and multiplication
defined by the biadditive extension of
\[
(\alpha y^ix^m)\cdot(\beta y^jx^n)
\colonequals \alpha\beta\gamma_m^{\,j}y^{i+j}x^{m+n}
\]
for \(\alpha,\beta\in K\) and \(i,j,m,n\in\mathbb N\).
Then \((S,x)\) is a GNOE of \(K[y]\) satisfying
\ref{it:ldivision}, and \((S,y)\) is a GNOE of \(K[x]\)
satisfying \ref{it:rdivision}.
Consequently, \(S\) is Noetherian.
\end{proposition}

\begin{proof}
For \(m\in\mathbb N\), let \(\sigma_m\) be the automorphism of
\(K[y]\) fixing \(K\) pointwise and satisfying
\(\sigma_m(y)=\gamma_m y\).
Define the additive bijection \(\theta\colon K[x]\to K[x]\) by
\(\theta(\sum_j \alpha_jx^j)\colonequals\sum_j \alpha_j\gamma_jx^j\).
For \(p,q\in K[y]\) and \(m,n\in\mathbb N\), we have
\[
(px^m)\cdot(qx^n)
=\bigl(p\sigma_m(q)\bigr)x^{m+n},
\]
whereas for \(p,q\in K[x]\),
\[
(y^m\cdot p)\cdot(y^n\cdot q)
=y^{m+n}\cdot\bigl(\theta^n(p)q\bigr).
\]

The monomial formula gives \ref{it:lg_ore1} and
\ref{it:rg_ore1} for both presentations.
The identities \(x^m\cdot x^n=x^{m+n}\) and
\(y^m\cdot y^n=y^{m+n}\) give
\ref{it:lg_ore2} and \ref{it:rg_ore2}.
The displayed identities and bijectivity of the coefficient maps
give \ref{it:lg_ore3} and \ref{it:rg_ore3}.
The monomial degree bounds give
\ref{it:lg_ore4} and \ref{it:rg_ore4}.

For \ref{it:ldivision} with respect to \((S,x)\), the proof of
\autoref{prop:nonORe-left-division} applies with \(\sigma_m\)
in place of \(\sigma^m\): these maps are surjective, and
right multiplication by powers of \(x\) shifts degrees without
changing coefficients.

Since \(K[x]\) is commutative, the second displayed identity
identifies \(S^{\mathrm{op}}\) with
\(K[x][y;\theta,0_{K[x]}]\).
As \(\theta\) is an additive bijection with \(\theta(1)=1\),
\autoref{prop:nonORe-left-division} gives \ref{it:ldivision}
for \((S^{\mathrm{op}},y)\), equivalently
\ref{it:rdivision} for \((S,y)\).
Since \(K[y]\) and \(K[x]\) are Noetherian,
\autoref{thm:hilbert} shows that \(S\) is Noetherian.
\end{proof}

If \(\gamma_{m+1}\neq\gamma_m\gamma_1\) for some \(m\in\mathbb N\),
then
\(
(x^m\cdot x)\cdot y-x^m\cdot(x\cdot y)
=(\gamma_{m+1}-\gamma_m\gamma_1)yx^{m+1}\neq0.
\)
Hence \(x\notin N_m(S)\). By
\autoref{prop:leftGNOE-NOE} and commutativity of \(K[y]\),
the presentation \((S,x)\) over \(K[y]\) is neither a
nonassociative Ore extension nor a flipped one. For example, if \(\gamma_m\colonequals\gamma^{m^2}\) with
\(\gamma\in K^\times\) and \(\gamma^2\neq1\), then
\(\gamma_2-\gamma_1^2=\gamma^2(\gamma^2-1)\neq0\).

\subsection{Poisson constructions}

Let \(K\) be an associative commutative ring.
A \emph{Poisson algebra} over \(K\) is an associative commutative
\(K\)-algebra \(A\) equipped with a \(K\)-bilinear Lie bracket
\(\{\cdot,\cdot\}\), called the \emph{Poisson bracket}, satisfying
\(\{a,bc\}=\{a,b\}c+b\{a,c\}\) for \(a,b,c\in A\). The \emph{depolarization} of \(A\) (see e.g. Markl and Remm~\cite{MR06})
is the nonassociative \(K\)-algebra with product
\(a\cdot b\colonequals ab+\{a,b\}\), using the normalization of
Loday and Vallette~\cite[Section~13.3.4]{LV12}.

We apply this construction to the differential case of the Poisson
polynomial extensions introduced by Oh \cite[Section~1]{Oh06}.
B\"ack, Lundstr\"om, \"Oinert, and Richter
\cite[Theorem~1]{BLOR26b} prove a version of Hilbert's Basis Theorem for Poisson
ideals of these extensions and their Laurent analogs.
Here we obtain left and right Noetherianity for the associated
nonassociative rings.

\begin{proposition}
\label{prop:poisson-differential}
Let \(R\) be an associative commutative ring and \(\delta\) a
derivation of \(R\). Equip \(A\colonequals R[x]\) with its ordinary
multiplication and the Poisson bracket determined by \(\{r,s\}=0\)
and \(\{x,r\}=\delta(r)\) for \(r,s\in R\).
Let \(S\) be the nonassociative ring obtained by depolarization
of \(A\), with multiplication
\(p\cdot q\colonequals pq+\{p,q\}\).
Then \((S,x)\) is a GNOE of \(R\) satisfying
\ref{it:ldivision} and \ref{it:rdivision}.
If \(R\) is Noetherian, then so is \(S\).
\end{proposition}

\begin{proof}
Extend \(\delta\) coefficientwise to \(A\). It commutes with
\(\partial_x\), so
\[\{p,q\}\colonequals
(\partial_x p)\delta(q)-\delta(p)(\partial_x q)\]
defines the stated Poisson bracket, viewing \(A\) as a
\(\mathbb Z\)-algebra.
For \(r,s\in R\) and \(m,n\in\mathbb N\),
\[
\{rx^m,sx^n\}
=\bigl(mr\delta(s)-n\delta(r)s\bigr)x^{m+n-1},
\]
where the right-hand side is zero when \(m=n=0\).
Thus \(R\) is a subring of \(S\), and \(1\) is its identity.
Moreover, \(x^m\cdot x^n=x^{m+n}\), and, with \(F_d\) as in
\autoref{prop:leading-term},
\[
(rx^m)\cdot(sx^n)
=rsx^{m+n}+\{rx^m,sx^n\}
\equiv rsx^{m+n}\pmod{F_{m+n-1}}.
\]
Applying \autoref{prop:leading-term} with \(\sigma\colonequals\id_R\) yields the desired result. 
\end{proof}

\begin{example}[A Noetherian non-Artinian GNOE]
\label{ex:poisson-nonsimple}
Let \(K\) be a field and equip \(A\colonequals K[y,x]\) with the
Poisson bracket \(\{p,q\}\colonequals
y((\partial_x p)(\partial_y q)
-(\partial_y p)(\partial_x q))\), so
\(\{x,y\}=y\).
Let \(S\) be the nonassociative ring obtained by depolarization
of \(A\), with multiplication
\(p\cdot q\colonequals pq+\{p,q\}\).

For \(j\geq1\), set \(I_j\colonequals y^jA\), using the ordinary
polynomial multiplication. Since
\(
\{p,y^jq\}
=y^j(j(\partial_x p)q+\{p,q\})\in I_j,
\)
each \(I_j\) is a Poisson ideal of \(A\), and hence a two-sided
ideal of \(S\). The strictly descending chain
\(S\supsetneq I_1\supsetneq I_2\supsetneq\cdots\) shows that
\(S\) is neither left nor right Artinian.
Also, \((x\cdot x)\cdot y-x\cdot(x\cdot y)=-y\neq0\), so
\(x\notin N_m(S)\). By \autoref{prop:leftGNOE-NOE} and commutativity of \(K[y]\), this excludes
both nonassociative Ore and flipped nonassociative Ore presentations
with this coefficient ring and distinguished element.
Nevertheless, taking \(R\colonequals K[y]\) and
\(\delta\colonequals y\partial_y\) in
\autoref{prop:poisson-differential} shows that \((S,x)\) is a
GNOE satisfying \ref{it:ldivision} and \ref{it:rdivision};
hence \(S\) is Noetherian.
\end{example}

\end{document}